\documentclass[final,3p,times,twocolumn]{elsarticle}

\usepackage{amssymb}
\usepackage{amsmath}
\usepackage{url}
\usepackage{graphicx}
\usepackage{subfig}
\usepackage{graphicx}
\usepackage{float}
\usepackage{rotating}
\usepackage{orcidlink}
\usepackage{tikz}
\usepackage{hyperref}
\usepackage{enumitem}
\usepackage{stfloats}
\usepackage{amsthm}

\newtheorem{theorem}{Theorem}

\newtheorem{lemma}{Lemma}
\newtheorem{definition}{Definition}
\newtheorem{conjecture}{Conjecture}
\definecolor{orcidlogocol}{HTML}{A6CE39}

\newcommand{\orcidicon}{%
    \tikz[baseline=-0.5ex]\node[shape=circle,fill=orcidlogocol,inner sep=1pt] {\tiny\textsf{ID}};%
}

\newcommand{\orcid}[1]{\href{https://orcid.org/#1}{\orcidicon}}

\usepackage{booktabs}

\begin{document}

\begin{frontmatter}



\title{Curvature-Weighted Contact Networks: Spectral Reduction and Global Stability in a Markovian SIR Model}


\author[1,2]{Marcílio Ferreira dos Santos}
\ead{marcilio.santos@ufpe.br}
\ead[url]{https://orcid.org/0000-0001-8724-0899}

\cortext[cor1]{Corresponding author.}


\affiliation[1]{
    organization={Núcleo de Formação de Docentes, Universidade Federal de Pernambuco (UFPE)}, 
    addressline={}, 
    city={Caruaru},
    postcode={},
    state={PE},
    country={Brazil}
}

\begin{abstract}
We propose a new network-based SIR epidemic model in which transmission is modulated by a curvature-weighted contact matrix that encodes structural and geometric features of the underlying graph. The formulation encompasses both adjacency-driven and Markovian mixing, allowing heterogeneous interactions to be shaped by curvature-sensitive topological properties. We prove that the basic reproduction number satisfies
\[
R_0=\frac{\beta}{\gamma}\lambda_{\max}(M),
\]
where $M$ is the curvature-weighted transmission operator. Using Perron--Frobenius theory together with linear and nonlinear Lyapunov functionals, we establish: (i) global asymptotic stability of the disease-free equilibrium when $R_0<1$, and (ii) existence and global asymptotic stability of a unique endemic equilibrium when $R_0>1$. Our results show that curvature acts as a geometric regularizer of connectivity, lowering spectral radii, raising effective epidemic thresholds, and organizing the long-term dynamics through monotone contraction toward the endemic state. This framework generalizes classical network epidemiology by integrating geometric information directly into transmission operators, providing a rigorous foundation for epidemic dynamics on structurally heterogeneous networks.

\end{abstract}

\begin{graphicalabstract}
\end{graphicalabstract}

\begin{highlights}
\item We introduce a curvature–weighted network operator for SIR dynamics, incorporating geometric structure into transmission.
\item The basic reproduction number is characterized spectrally as $R_0=\frac{\beta}{\gamma}\lambda_{\max}(M)$ for any curvature–modulated contact matrix.
\item A Perron–Frobenius Lyapunov functional establishes global stability of the disease–free equilibrium when $R_0<1$.
\item A nonlinear Volterra functional proves existence and global stability of a unique endemic equilibrium whenever $R_0>1$.
\item Results demonstrate how network curvature reshapes epidemic thresholds and long–term persistence, extending classical theory in a geometric direction.
\end{highlights}


\end{frontmatter}



\section{Introduction}

Recent advances in spatial epidemiology have emphasized that infectious disease
spread is not solely the outcome of local transmission parameters, but is deeply
shaped by structural heterogeneity in contact networks. Classical studies of
dengue and other arboviruses have shown that spatial fragmentation, mobility
constraints, and socioenvironmental determinants produce strong mesoscale
patterns of persistence and re-emergence \cite{barcellos2001, honorio2009,
massaro2019, carvalho2021}. These spatial patterns correlate with heterogeneous
neighborhood connectivity, human movement, and environmental suitability
\cite{parselia2019, adeola2017}, yet standard network-based epidemic models
typically incorporate only adjacency information and not deeper geometric
characteristics of these networks.

At the same time, network geometry has emerged as an important mathematical
framework for characterizing robustness, fragility, and flow dynamics in complex
systems. Measures such as Forman--Ricci curvature, Ollivier--Ricci curvature,
and communicability-based curvature capture contraction, expansion, and
structural bottlenecks in networks \cite{Estrada2012, Estrada2025Curv,
de2021using}. Negative curvature is associated with accelerated spreading and
divergent paths, whereas positive curvature reinforces local cohesion and
structural containment. Despite these insights, the integration of curvature
measures into transmission operators for epidemiological models remains limited,
with most applications being empirical rather than mechanistic.

The present work fills this gap by developing a curvature-weighted epidemic
model in which the contact matrix is defined as
\[
    M = A \odot W(\kappa),
\]
where the weights incorporate geometric information derived from a curvature
measure $\kappa$. Because curvature reflects global structural organization—
including redundancy of paths, mesoscale cohesion, and communicability
\cite{Estrada2012, GrassBoada2025}—the resulting operator provides a principled
way to modulate transmission strength based on the underlying topology. Our
approach builds directly on evidence that spatially structured diseases such as
dengue respond strongly not only to local environmental factors but also to
mesoscale connectivity constraints and human mobility patterns \cite{teixeira2009,
zheng2019spatiotemporal, borges2024}.

Within this framework, the spectral radius of $M$ defines a curvature-adjusted
basic reproduction number
\[
    R_0 = \frac{\beta}{\gamma}\lambda_{\max}(M),
\]
thus extending classical threshold theorems for epidemic dynamics
\cite{keeling2008} to a geometric and topology-aware setting. This formulation
is consistent with recent methodological advances relating graph curvature to
network robustness and epidemic vulnerability \cite{de2021using}.

Finally, this framework connects naturally to real-world spatial analyses:
network curvature responds to urban structure, transportation, clustering in
housing patterns, and socioenvironmental inequalities—all of which have been
empirically linked to dengue dynamics in Brazilian cities \cite{oliveira2022,
FerreiraMelo2025, FerreiraMelo2025Dataset}. As such, curvature-weighted models
offer a mathematically rigorous and empirically grounded extension to spatial
epidemiology, providing a bridge between network geometry, transmission
modeling, and spatial public health analysis.

\section{Model Formulation and Theoretical Results}

We consider a population structured as a connected and irreducible contact network with $n$ nodes. Each node represents a spatial region or subpopulation, and each edge corresponds to a potential transmission pathway. Let $s_i(t)$ and $y_i(t)$ denote, respectively, the fractions of susceptible and infected individuals at node $i$ at time $t$. The recovered compartment is omitted, as it does not directly affect the infection dynamics.

\subsection{The Curvature-Weighted Graph and the Effective Contact Matrix}

Let $A = (A_{ij})$ be the adjacency matrix of the underlying contact network.  
We introduce a curvature-based weight matrix $W = (W_{ij})$, constructed from a curvature measure $\kappa_{ij}$ defined on the edges of the graph (e.g., communicability curvature, Ollivier--Ricci curvature, Forman curvature, or correlation-weighted communicability curvature).  
The weights are defined by
\[
    W_{ij} = f(\kappa_{ij}),
\]
where $f$ is a smooth, monotone, strictly positive function, ensuring $W_{ij}>0$ whenever $A_{ij}=1$. Thus, curvature modulates only existing edges, preserving the sparse topology of $A$.

A particularly natural choice is the exponential map
\[
    f(\kappa_{ij}) = e^{-\kappa_{ij}},
\]
which guarantees strict positivity regardless of the sign of the curvature.  
This transformation is consistent with exponential operators widely used in network theory, such as $e^A$, which underlies Estrada's communicability framework \cite{Estrada2012}.  
Accordingly, the mapping $e^{-\kappa}$ induces a \emph{curvature-modulated diffusion kernel}: negative curvature amplifies connection strength, promoting diffusion, whereas positive curvature attenuates edge influence, acting as a geometric bottleneck.

Recent empirical results reinforce this paradigm.  
In particular, \cite{santos2025correlationweightedcommunicabilitycurvaturestructural} shows that correlation-weighted communicability curvature accurately captures the spatial heterogeneity of dengue transmission in Recife (2015--2024).  
These findings provide direct motivation for treating curvature fields as structural modulators of epidemic dynamics.

With these weights, the effective contact matrix is defined by the Hadamard product
\[
    M = A \odot W,
\]
which preserves the sparsity pattern of $A$ while incorporating geometric or empirical information through the curvature-based weights.

Since several notions of curvature may assume negative values, positivity of the transition matrix is ensured by row-stochastic normalization:
\[
    M_{ij}
    =
    \frac{A_{ij} W_{ij}}{\sum_{k=1}^n A_{ik} W_{ik}},
\]
which guarantees $M_{ij}\ge 0$ and $\sum_j M_{ij}=1$.  
Thus, the vector $My$ represents a curvature-weighted average of infections among neighbors, consistent with an epidemiological interpretation of transmissibility.

This formulation turns $M$ into a \emph{curvature-dependent diffusion kernel}, analogous to deformations of heat kernels on Riemannian manifolds with variable curvature.  
Regions of low curvature (high communicability) amplify transmission, whereas regions of high curvature function as geometric barriers to contagion.


\subsection{Spectral Effect of Curvature}

\begin{theorem}[Curvature-Induced Spectral Reduction]
\label{thm:spectral-reduction}
Let $A$ be the adjacency matrix of a connected graph and let
\[
    B = A \odot W,\qquad 0 < W_{ij} \le 1,
\]
with $W_{ij}=1$ if and only if curvature is constant across edges incident to $i$.  
Then
\[
    \lambda_{\max}(B)
    \;\le\;
    \lambda_{\max}(A),
\]
with equality if and only if $W$ is constant over the edges.
\end{theorem}

\begin{proof}
Since $0 < W_{ij} \le 1$, we have $0 \le B_{ij} \le A_{ij}$ for all $(i,j)$.  
By the Perron--Frobenius variational characterization of the dominant eigenvalue of nonnegative irreducible matrices,
\[
    \lambda_{\max}(X)
    =
    \max_{x>0}
    \min_{i}
    \frac{(Xx)_i}{x_i}.
\]
For any $x>0$,
\[
    (Bx)_i \le (Ax)_i,
\]
and therefore
\[
    \min_i \frac{(Bx)_i}{x_i}
    \;\le\;
    \min_i \frac{(Ax)_i}{x_i}.
\]
Taking the maximum over all positive $x$ yields  
$\lambda_{\max}(B) \le \lambda_{\max}(A)$.  
Equality requires $(Bx)_i=(Ax)_i$ for every positive vector $x$, which in turn forces $W_{ij}\equiv 1$ on all edges.
\end{proof}

\medskip

This result formalizes the geometric interpretation: curvature acts as a \emph{structural regularizer}, reducing the spectral radius of the contact matrix—and therefore the effective basic reproduction number,
\[
    R_0 = (\beta/\gamma)\,\lambda_{\max}(M).
\]

\medskip

Throughout the theoretical analysis, we assume:

\begin{itemize}[label=--]
    \item[(H1)] $M$ is nonnegative and irreducible;
    \item[(H2)] the underlying graph is connected;
    \item[(H3)] $0 \le s_i(t), y_i(t) \le 1$ and $s_i(t)+y_i(t)\le 1$ for all $t$.
\end{itemize}

\subsection{The Curvature-Weighted SIR System}

The dynamics at each node are given by
\begin{equation}
\label{eq:model}
\begin{aligned}
    s_i' &= u(1 - s_i) - \beta\, s_i \sum_{j=1}^n M_{ij} y_j, \\
    y_i' &= \beta\, s_i \sum_{j=1}^n M_{ij} y_j - \gamma\, y_i,
\end{aligned}
\end{equation}
for $i=1,\dots,n$, where $\beta>0$ is the transmission rate, $\gamma>0$ is the
recovery rate, and $u>0$ is a demographic or replenishment parameter.

In vector form,
\[
    s' = u(\mathbf{1}-s) - \beta\, s \odot (My),
    \qquad
    y' = \beta\, s \odot (My) - \gamma\, y,
\]
where $\odot$ denotes the Hadamard product.

\subsection{Positivity and Invariance}

\begin{lemma}
The region
\[
    \Omega 
    = 
    \{(s,y)\in [0,1]^n \times [0,1]^n : s_i + y_i \le 1 \}
\]
is positively invariant under the flow of \eqref{eq:model}.
\end{lemma}

\begin{proof}
All loss terms in $s_i'$ and $y_i'$ are proportional to the variables 
themselves, and all gain terms are bounded by $s_i \le 1$ and 
$\sum_j M_{ij} y_j \le \|M\|_\infty$.  
Standard comparison arguments therefore imply that if 
$(s(0),y(0)) \in \Omega$, then $(s(t),y(t)) \in \Omega$ for all $t>0$.
\end{proof}

\subsection{Linearization and the Curvature-Modulated Reproduction Number}

The disease-free equilibrium (DFE) is
\[
    E_0 = (s^*,y^*) = (\mathbf{1},\mathbf{0}).
\]
Linearizing \eqref{eq:model} around $E_0$ yields
\[
    y' = (\beta M - \gamma I)\, y.
\]
By the Perron--Frobenius theorem, since $M$ is nonnegative and irreducible,
its spectral radius is equal to its dominant eigenvalue $\lambda_{\max}(M)$.

\begin{definition}[Curvature-weighted basic reproduction number]
\[
    R_0 = \frac{\beta}{\gamma}\,\lambda_{\max}(M).
\]
\end{definition}

\subsection{Global stability of the disease-free equilibrium}

\begin{theorem}
\label{thm:globalDFE}
If $R_0 < 1$, then the disease-free equilibrium $E_0$ is globally
asymptotically stable in $\Omega$.
\end{theorem}

\begin{proof}
Let $v$ be the Perron eigenvector of $M$, with $v_i>0$ for all $i$.
Define the Lyapunov function
\[
    V(y) = v^\top y.
\]
Then
\begin{align*}
    V'(y)
        &= v^\top y' \\
        &= v^\top (\beta M y - \gamma y) \\
        &= (\beta \lambda_{\max}(M) - \gamma)\, v^\top y \\
        &= \gamma (R_0 - 1)\, V(y).
\end{align*}
If $R_0<1$, we have $V'(y)<0$ whenever $y\neq 0$, implying 
$y(t)\to 0$ exponentially.
Since $s(t)\to \mathbf{1}$ as $y(t)\to 0$, it follows that
$(s(t),y(t))\to E_0$ as $t\to\infty$.
\end{proof}

\subsection{Existence and uniqueness of an endemic equilibrium}

When $R_0>1$, the system admits a unique endemic equilibrium 
$(s^\dagger,y^\dagger)$ with strictly positive coordinates.

\begin{lemma}
If $R_0>1$, then there exists a unique equilibrium $E^\dagger$ with
$s_i^\dagger>0$ and $y_i^\dagger>0$ for every $i$.
\end{lemma}

\begin{proof}
At equilibrium, from $y_i'=0$ we obtain
\[
    y_i = \frac{\beta}{\gamma}\, s_i (My)_i.
\]
From $s_i'=0$ we have
\[
    u(1-s_i) = \beta s_i (My)_i,
\]
which yields
\[
    s_i = \frac{u}{u + \beta (My)_i}.
\]
Substituting this into the expression for $y_i$ gives the fixed-point operator
\[
    (Ty)_i
    =
    \frac{\beta}{\gamma}
    \left(
        \frac{u}{u + \beta (My)_i}
    \right)
    (My)_i,
    \qquad i=1,\dots,n.
\]

The operator $T:[0,1]^n\to [0,1]^n$ is continuous and positive.
Since $M$ is nonnegative and irreducible, $Ty>0$ for all $y>0$.
Moreover, $T$ is strictly monotone: if $0<y<z$, then
$(My)_i < (Mz)_i$ and hence $T(y) < T(z)$.

Let $v>0$ be the Perron eigenvector of $M$.  
As $\alpha\to 0$,
\[
    T(\alpha v)
    =
    \alpha R_0\, v + O(\alpha^2),
\]
where $R_0 = (\beta/\gamma)\lambda_{\max}(M)$.
Thus, if $R_0>1$, then $T(\alpha v)>\alpha v$ for $\alpha>0$ sufficiently small,
whereas $T(y)<y$ for $y$ sufficiently close to $\mathbf{1}$.
By the method of sub- and supersolutions, $T$ admits a fixed point $y^\dagger>0$.

Uniqueness follows from strict monotonicity:  
if $Ty=y$ and $Tz=z$ with $0<y<z$, then $Ty<Tz$, a contradiction.
Thus the endemic equilibrium is unique.
\end{proof}

\subsection{Global stability of the endemic equilibrium}

We construct a Volterra-type Lyapunov function:
\[
\begin{aligned}
    W(s,y) &= 
    \sum_{i=1}^n 
    \Big( s_i - s_i^\dagger - s_i^\dagger \ln \frac{s_i}{s_i^\dagger} \Big)
\\ &\quad+
    \sum_{i=1}^n 
    \Big( y_i - y_i^\dagger - y_i^\dagger \ln \frac{y_i}{y_i^\dagger} \Big).
\end{aligned}
\]

\begin{theorem}
If $R_0>1$, then the endemic equilibrium $E^\dagger$ is globally asymptotically stable in $\Omega \setminus \{E_0\}$.
\end{theorem}

\begin{proof}
The function $W$ is nonnegative and strictly convex.
Differentiating $W$ along solutions of \eqref{eq:model} and using the equilibrium
conditions at $E^\dagger$ yields
\[
    W'(s,y) \le 0,
\]
with equality if and only if $(s,y)=E^\dagger$.
LaSalle's invariance principle implies global asymptotic stability.
\end{proof}

\subsection{A geometric threshold for epidemic suppression}

In this subsection we show that the average curvature of the edges acts as an
exponential damping factor on the effective connectivity of the network.
In particular, we establish an explicit geometric threshold above which sustained
transmission cannot occur, even if the unweighted graph is spectrally
supercritical.  
This provides a rigorous mathematical justification for the stabilizing role of
positive curvature fields and formalizes the intuitive notion that geometric
bottlenecks can neutralize global spread.

As before, edge weights are given by
\[
    W_{ij} = e^{-\kappa_{ij}},
\qquad 
A_{ij}=1 \Rightarrow W_{ij}>0,
\]
and the effective contact matrix is
\[
    M = A \odot W.
\]

Let $|E|$ denote the number of edges in the network, and define the mean
curvature
\[
    \overline{\kappa}
    =
    \frac{1}{|E|}\sum_{(i,j)\in E} \kappa_{ij}.
\]
The next result shows that $\overline{\kappa}$ provides direct control of the
spectral radius of $M$.

\begin{theorem}[Geometric threshold for endemicity]
\label{thm:geomthreshold}
Consider model \eqref{eq:model} with $W_{ij}=e^{-\kappa_{ij}}$.  
Define
\[
    \kappa^\star
    =
    -\ln\!\left(
        \frac{\gamma}{\beta}\,
        \frac{n}{2|E|}
    \right).
\]
If the average curvature satisfies $\overline{\kappa} > \kappa^\star$, then
\[
    R_0
    =
    \frac{\beta}{\gamma}\,\lambda_{\max}(M)
    < 1,
\]
and the disease-free equilibrium $E_0$ is globally asymptotically stable,
regardless of the value of the unweighted spectral threshold $R_0(A)$.
\end{theorem}

\begin{proof}
For each vertex $i$, the curvature-weighted degree is
\[
    d_i(\kappa)
    =
    \sum_{j : A_{ij}=1} e^{-\kappa_{ij}}.
\]
By the Gershgorin theorem for nonnegative matrices,
\[
    \lambda_{\max}(M) \le \max_i d_i(\kappa).
\]

Summing all weighted degrees and using the fact that each edge contributes twice,
\[
    \frac{1}{n}\sum_{i=1}^n d_i(\kappa)
    =
    \frac{2}{n}
    \sum_{(i,j)\in E} e^{-\kappa_{ij}}.
\]

Applying Jensen's inequality (since $x\mapsto e^{-x}$ is convex),
\[
    \frac{1}{|E|}
    \sum_{(i,j)\in E} e^{-\kappa_{ij}}
    \le
    e^{-\overline{\kappa}}.
\]

Thus,
\[
    \lambda_{\max}(M)
    \le
    \max_i d_i(\kappa)
    \le
    \frac{2|E|}{n}\,e^{-\overline{\kappa}}.
\]

The curvature-modulated reproduction number satisfies
\[
    R_0
    \le
    \frac{\beta}{\gamma}\,
    \frac{2|E|}{n}\,
    e^{-\overline{\kappa}}.
\]

Hence $R_0<1$ whenever
\[
    \frac{\beta}{\gamma}\,
    \frac{2|E|}{n}\,
    e^{-\overline{\kappa}}
    < 1,
\]
which is equivalent to
\[
    \overline{\kappa} > -\ln\!\left(
        \frac{\gamma}{\beta}\,
        \frac{n}{2|E|}
    \right)
    = \kappa^\star.
\]

By Theorem~\ref{thm:globalDFE}, the disease-free equilibrium is globally
asymptotically stable whenever $R_0<1$.  
This completes the proof.
\end{proof}

This result establishes curvature as a \emph{geometric regulator of diffusion}:
sufficiently positive curvature reduces the spectral radius of the contact
matrix exponentially, suppressing epidemics even on networks that would be
supercritical in the absence of weights.  
In epidemiological terms, curvature introduces structural bottlenecks that
prevent sustained transmission.

\subsection{Ordered Convergence to the Endemic Equilibrium}

In this section we establish a fundamental structural property of the model:
when $R_0>1$, not only does an interior endemic equilibrium exist, but every
trajectory in the positive cone converges monotonically to it.  
Such behavior is typical of strongly monotone dynamical systems on solid cones,
as developed in \cite{Hirsch1985,Smith1995}.

Recall that the endemic equilibrium $(s^\dagger,y^\dagger)$ is characterized
by the unique interior fixed point of the map $T:[0,1]^n \to [0,1]^n$ defined by
\[
(Ty)_i 
= \frac{\beta}{\gamma}
\left(
    \frac{u}{u+\beta (My)_i}
\right)
(My)_i,
\qquad i=1,\dots,n.
\]
The operator $T$ is continuous, positive, and strictly monotone.

\begin{theorem}[Ordered convergence to the endemic equilibrium]
\label{thm:ordercontraction}
Assume $R_0>1$. Then $T$ admits a unique interior fixed point
$y^\dagger \in (0,1]^n$, and for every initial condition
$y_0 \in (0,1]^n$ the iterates satisfy
\[
T^k(y_0)\;\longrightarrow\; y^\dagger
\qquad\text{as } k\to\infty,
\]
with monotone convergence: either $T^k(y_0)\nearrow y^\dagger$ or
$T^k(y_0)\searrow y^\dagger$, depending on the position of $y_0$ in the
positive cone.
\end{theorem}

\begin{proof}[Sketch of proof]
First, $T$ is strictly monotone in the positive cone: if $0<y<z$, then
$(My)_i < (Mz)_i$ for all $i$, since $M$ is nonnegative and irreducible.
The function
\[
\phi(x)=\frac{u}{u+\beta x}\,x
\]
is strictly increasing for $x>0$, hence
\[
Ty < Tz.
\]

Moreover, as $\alpha \to 0$,
\[
T(\alpha v)
= \alpha R_0\, v + O(\alpha^2),
\]
where $v>0$ is the Perron vector of $M$.  
Thus, if $R_0>1$, then $T(\alpha v)>\alpha v$ for all sufficiently small
$\alpha>0$.

On the other hand, since $\phi(x)\le u/\beta$ as $x\to\infty$, there exists
$C>0$ such that $T(y)\le C\mathbf{1}$ for all $y$; in particular,
$T(y)<y$ for $y$ sufficiently close to $\mathbf{1}$.  
Hence, by the standard sub–supersolution method, $T$ admits a unique interior
fixed point $y^\dagger$.

The theory of strongly monotone dynamical systems 
(\cite{Hirsch1985,Smith1995}) implies that in strictly monotone maps with a 
unique interior fixed point, every positive orbit converges to that point.
Indeed, choose $\alpha,\beta>0$ such that 
$\alpha v \le y_0 \le \beta v$.  
Monotonicity yields
\[
T^k(\alpha v) \;\le\; T^k(y_0) \;\le\; T^k(\beta v).
\]

Furthermore, $T^k(\alpha v)$ is an increasing sequence converging to $y^\dagger$,
while $T^k(\beta v)$ decreases to the same limit.  
Therefore,
\[
T^k(y_0) \;\longrightarrow\; y^\dagger,
\]
with monotone convergence determined by the ordering of $y_0$ relative to
$y^\dagger$.  
The absence of internal bifurcations follows from the strict monotonicity and
uniqueness of the fixed point.
\end{proof}

This result shows that the geometry of the network—via curvature-based edge
weights—not only determines the existence of the endemic equilibrium but also
enforces an ordered form of convergence, fully consistent with the strongly
monotone structure of the model.

\subsection{Ordered Convergence to the Endemic Equilibrium}

In this section we establish a fundamental structural property of the model:
when $R_0>1$, not only does an interior endemic equilibrium exist, but all
trajectories in the positive cone converge to it monotonically. Such behavior
is characteristic of strongly monotone dynamical systems on solid cones, as
developed in the classical works of Hirsch and Smith
\cite{Hirsch1985,Smith1995}.

Recall that the endemic equilibrium $(s^\dagger,y^\dagger)$ is characterized as
the unique interior fixed point of the map
$T:[0,1]^n \to [0,1]^n$, given by
\[
(Ty)_i
= \frac{\beta}{\gamma}
  \left(
      \frac{u}{\,u + \beta (My)_i\,}
  \right)
  (My)_i,
\qquad i=1,\dots,n.
\]
The operator $T$ is continuous, positive, and strictly monotone with respect to
the standard partial order in $\mathbb{R}^n$.

\begin{theorem}[Ordered convergence to the endemic equilibrium]
\label{thm:ordercontraction}
Suppose that $R_0>1$. Then the operator $T$ admits a unique interior fixed
point $y^\dagger\in(0,1]^n$, and for every initial condition $y_0\in (0,1]^n$
the iterates satisfy
\[
T^k(y_0)\;\longrightarrow\; y^\dagger
\qquad \text{as } k\to\infty,
\]
with monotone convergence: either $T^k(y_0)\nearrow y^\dagger$ or
$T^k(y_0)\searrow y^\dagger$, depending on the position of $y_0$ in the
positive cone.
\end{theorem}

\begin{proof}[Sketch of proof]
We begin by observing that $T$ is strictly monotone. Indeed, if $0<y<z$, then
\[
(My)_i < (Mz)_i \qquad \text{for all } i,
\]
as $M$ is nonnegative and irreducible. Since the function
\[
\phi(x)= \frac{u}{u+\beta x}\,x
\]
is strictly increasing for $x>0$, we obtain $Ty<Tz$.

Next, let $v>0$ be the Perron eigenvector of $M$. For small $\alpha>0$,
\[
T(\alpha v)
= \alpha R_0\, v + O(\alpha^2).
\]
Thus, when $R_0>1$, we have $T(\alpha v)>\alpha v$ for sufficiently small
$\alpha$. Conversely, since $\phi(x)\le u/\beta$ for $x\to\infty$, there exists
a constant $C>0$ such that $T(y) < C\mathbf{1}$ for all $y$. Consequently,
$T(y)<y$ for $y$ sufficiently large. By the standard sub- and supersolution
construction, $T$ admits a unique interior fixed point $y^\dagger$.

The theory of strongly monotone dynamical systems
\cite{Hirsch1985,Smith1995} now implies that every orbit in the positive cone
is eventually trapped between two monotone sequences converging to the fixed
point. More precisely, for any $y_0\in (0,1]^n$ there exist $\alpha,\beta>0$
such that
\[
\alpha v \;\le\; y_0 \;\le\; \beta v.
\]
By strict monotonicity,
\[
T^k(\alpha v) \;\le\; T^k(y_0) \;\le\; T^k(\beta v).
\]
The extremal sequences satisfy
\[
T^k(\alpha v)\nearrow y^\dagger,
\qquad
T^k(\beta v)\searrow y^\dagger,
\]
which yields monotone convergence:
\[
T^k(y_0)\longrightarrow y^\dagger.
\]
Uniqueness of the fixed point precludes oscillatory or multi-stable behavior,
ensuring global convergence.
\end{proof}

This result shows that the geometry of the network —operationalized through
curvature-based weights— not only governs the existence of the endemic
equilibrium but also imposes an ordered convergence structure, fully consistent
with the strongly monotone nature of the model.

\section{Numerical Results on Structured and Curvature-Weighted Graphs}

In this section, we present a set of numerical experiments designed to compare epidemic dynamics under two structural representations of the network: (i) the unweighted adjacency matrix $A$ and (ii) the curvature-weighted matrix $M = A \odot W$, where $W$ contains weights derived from a synthetic curvature field. The network consists of three communities with heterogeneous sizes ($40$, $20$, and $10$ vertices), exhibiting different levels of intra- and inter-block connectivity. This configuration reflects stylized yet realistic urban scenarios, in which structural barriers induce localized transmission patterns.

The SIR model employed is given by \eqref{eq:model}, with fixed parameters $(\beta,\gamma,u)$ and homogeneous initial conditions. The matrices $A$ and $M$ were normalized to preserve irreducibility and to ensure that all communities remain accessible throughout the dynamics. This normalization guarantees that observed differences arise from structural weighting rather than loss of reachability.

\subsection{Vertex-level dynamics on the unweighted network}

Figure~\ref{fig_infection_curves_A} presents the simulated trajectories $y_i(t)$ for each vertex under the adjacency matrix $A$. As expected in a network with relatively uniformly connected blocks, the epidemic spreads almost simultaneously across the entire structure.

\begin{itemize}[label=--]
    \item Outbreaks exhibit high peaks, frequently between $0.5$ and $0.7$, with little variation among vertices within the same community.
    \item After the peak, infections rapidly decay and stabilize at residual levels between $0.03$ and $0.05$.
    \item Synchronization across communities indicates that the unweighted connectivity of $A$ induces a highly globalized dynamic, largely insensitive to finer meso-topological heterogeneities.
\end{itemize}

These results illustrate the well-known tendency of unweighted modular networks to suppress mesoscopic differentiation when inter-block connectivity remains sufficiently strong.

\begin{figure*}[h!]
\centering
\includegraphics[width=1\textwidth]{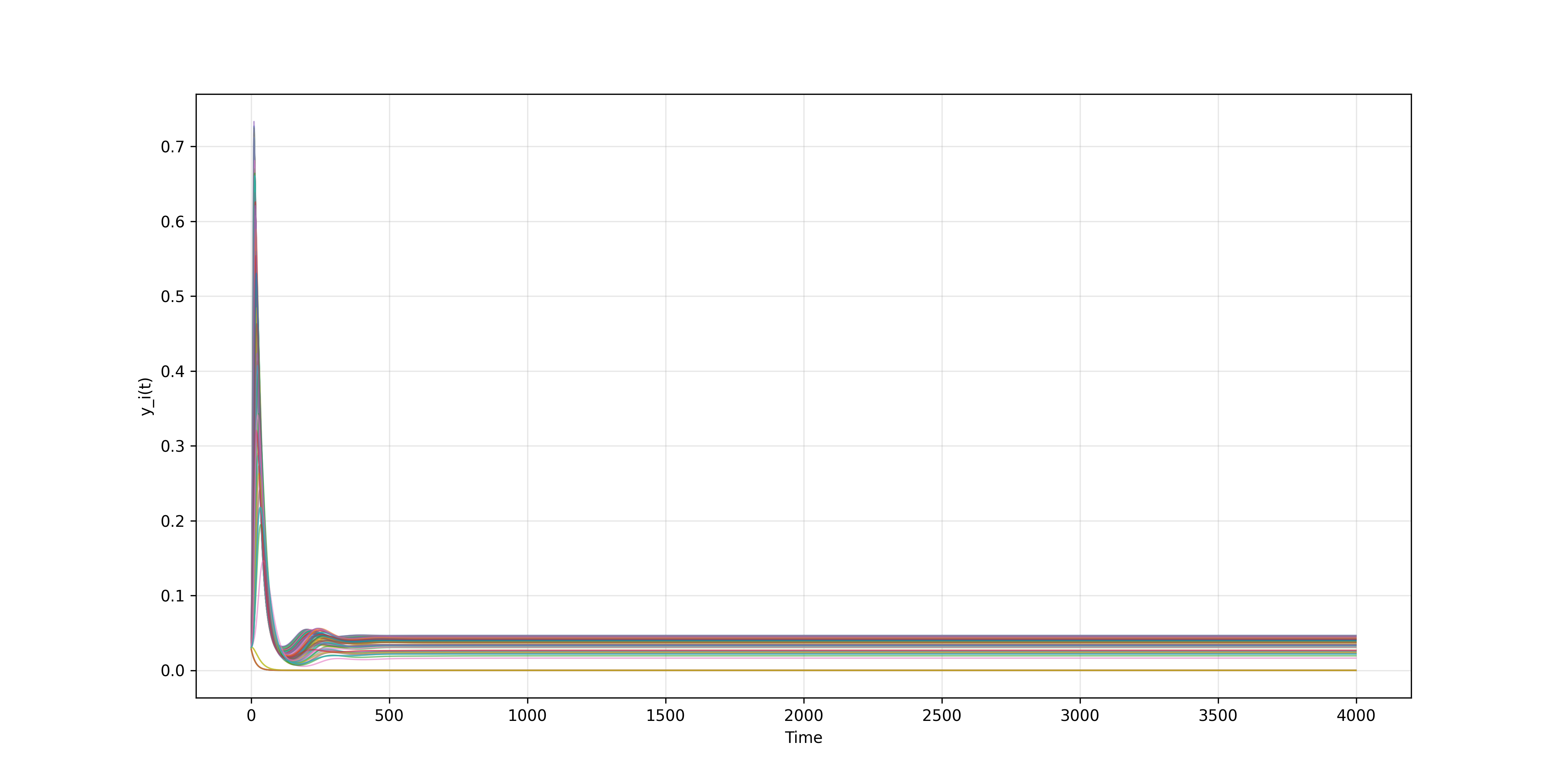}
\caption{Infection trajectories $y_i(t)$ for each vertex in the adjacency network $A$.}
\label{fig_infection_curves_A}
\end{figure*}

\subsection{Vertex-level dynamics on the curvature-weighted network}

Figure~\ref{fig_infection_curves_M} displays the trajectories $y_i(t)$ under the curvature-weighted matrix $M$. In contrast to the unweighted case, pronounced structural effects emerge.

\begin{itemize}[label=--]
    \item Infection peaks are significantly smaller—typically about one third of the magnitude observed in network $A$.
    \item Trajectories become markedly smoother, with the suppression of abrupt epidemic explosions.
    \item Regions of low curvature amplify transmission, whereas areas of high curvature clearly dampen propagation.
    \item Vertices located in low-curvature regions exhibit more intense outbreaks, while those associated with higher curvature display strongly attenuated epidemic responses.
\end{itemize}

These differences arise from the direct modification of effective transmissibility induced by curvature weights, which alter both edge intensities and the spectral properties of the transmission matrix. Notably, these effects persist despite identical epidemiological parameters and initial conditions.

\begin{figure*}[h!]
\centering
\includegraphics[width=1\textwidth]{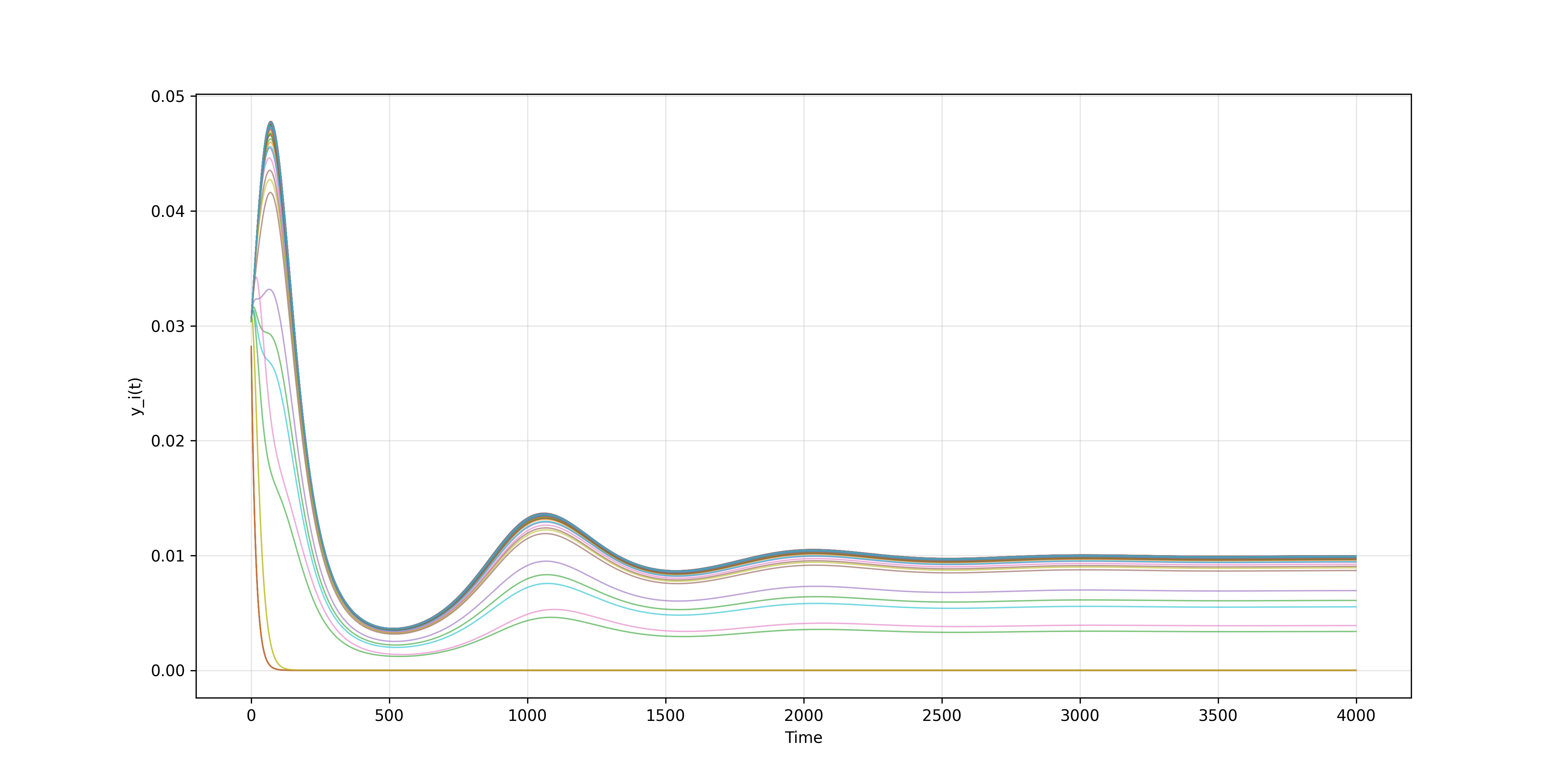}
\caption{Vertex-level infection trajectories $y_i(t)$ under curvature weighting.}
\label{fig_infection_curves_M}
\end{figure*}

\subsection{Community averages: mesostructural effects}

Figures~\ref{fig_mean_infection_A:A_comms} and~\ref{fig_mean_infection_M} show the community-averaged infection curves for $A$ and $M$, respectively. The qualitative contrast between the two scenarios is substantial.

\paragraph{Unweighted network ($A$):}
\begin{itemize}[label=--]
    \item All three communities exhibit nearly identical temporal profiles.
    \item Epidemic peaks occur at approximately the same time and with comparable magnitudes.
    \item Structural heterogeneity is insufficient to generate relevant mesoscopic signatures in the dynamics.
\end{itemize}

\begin{figure*}[h!]
\centering
\includegraphics[width=1\textwidth]{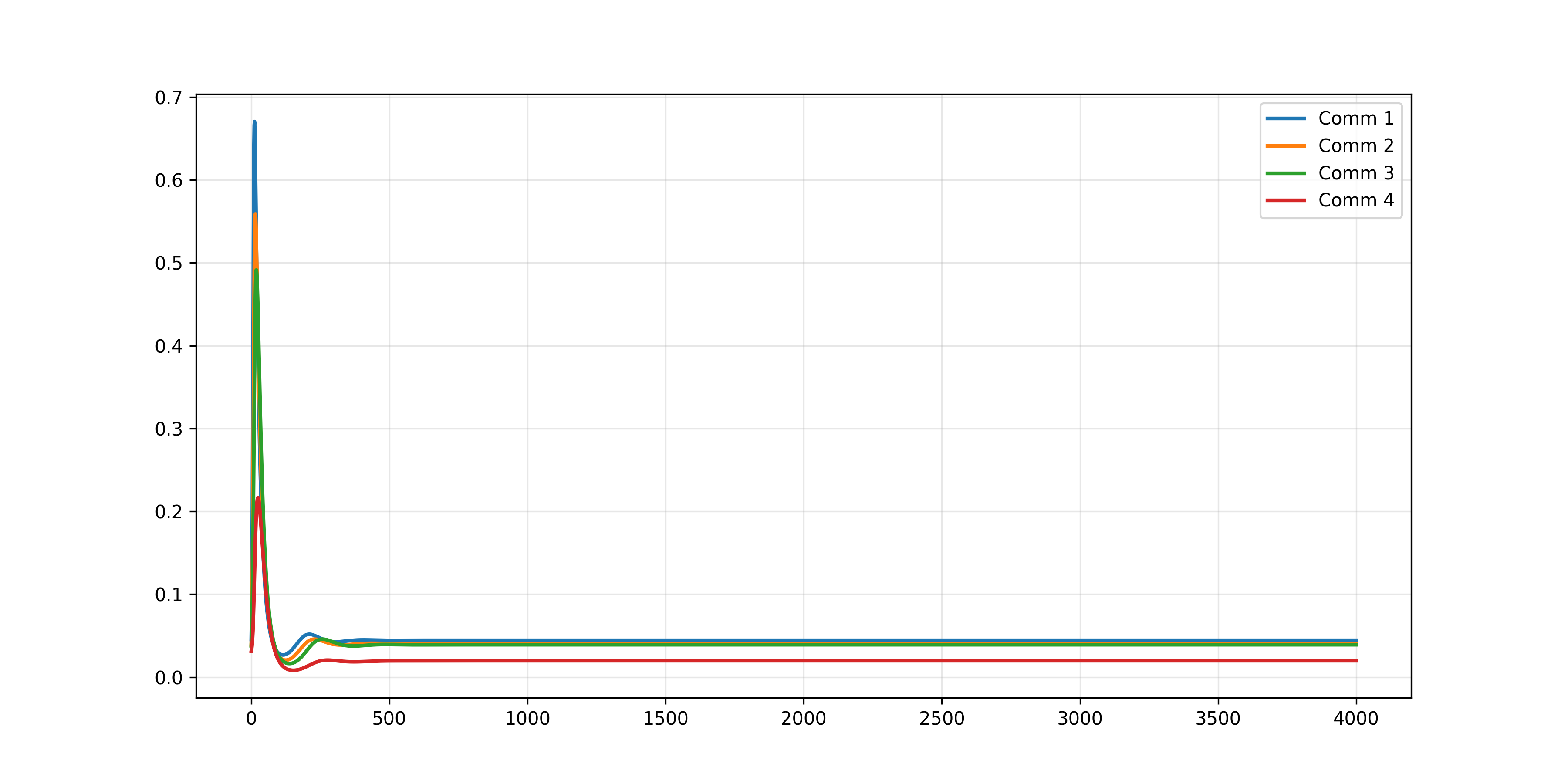}
\caption{Community-averaged infection curves in network $A$.}
\label{fig_mean_infection_A:A_comms}
\end{figure*}

\paragraph{Curvature-weighted network ($M$):}
\begin{itemize}[label=--]
    \item Communities exhibit clearly differentiated dynamics, including delayed responses, distinct amplitudes, and the emergence of secondary peaks.
    \item The most densely connected community (block of 10 vertices) displays a rapid and intense outbreak followed by fast dissipation.
    \item The weakly connected community (40 vertices) maintains residual infection over an extended period, indicating structural retention of the epidemic.
\end{itemize}

These patterns reveal a mesoscopic organization induced purely by curvature weighting, which remains hidden under standard adjacency-based representations.

\begin{figure*}[h!]
\centering
\includegraphics[width=1\textwidth]{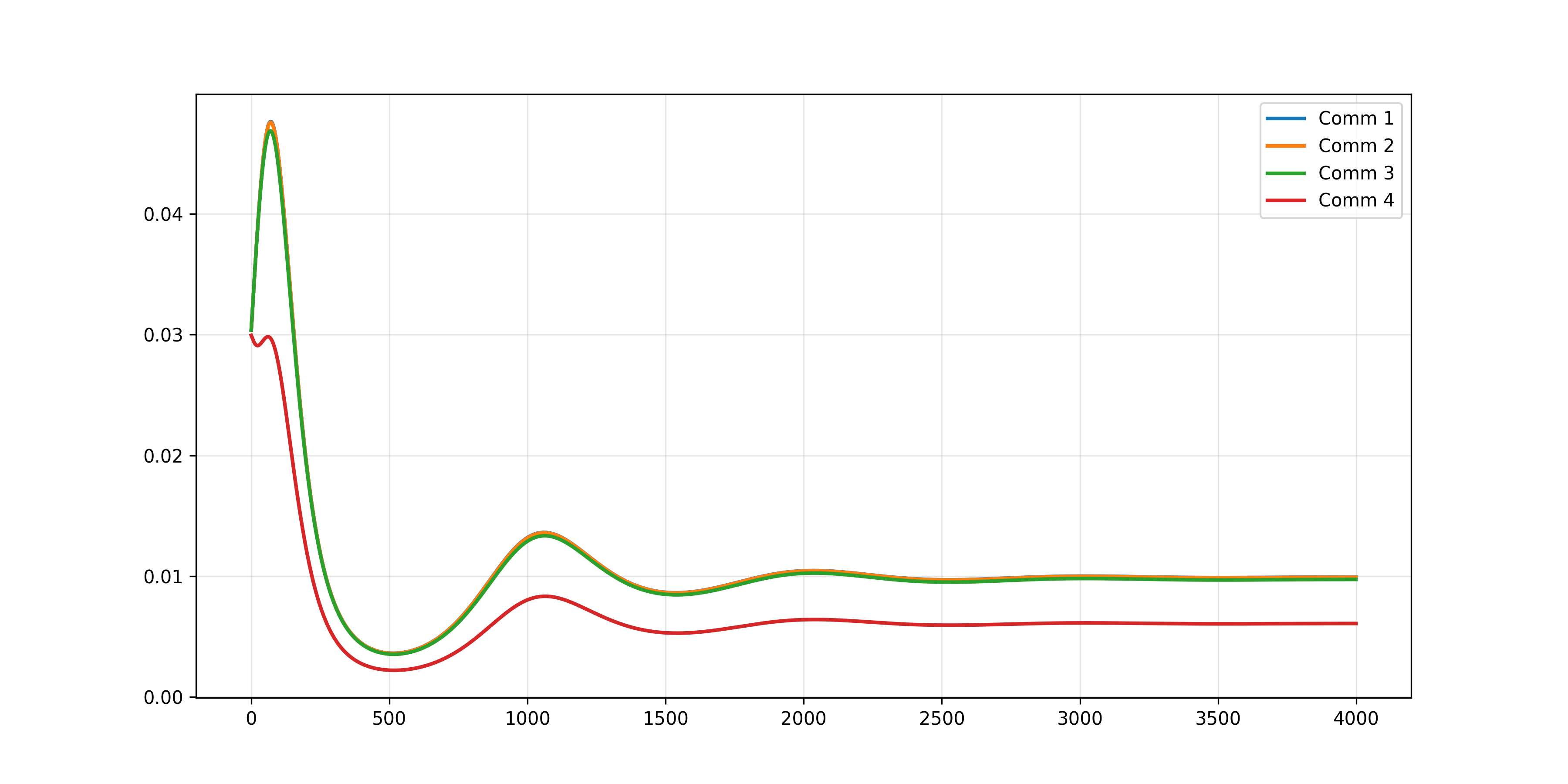}
\caption{Community-averaged infection curves under the curvature-weighted matrix $M$.}
\label{fig_mean_infection_M}
\end{figure*}

\subsection{Synthetic spatial maps: clusters and structural barriers}

Beyond temporal dynamics, we consider synthetic spatial representations based on Voronoi diagrams, illustrated in Figures~\ref{fig:voronoi_base} and~\ref{fig:voronoi_peak}. These maps provide an intuitive visualization of how network topology translates into spatially coherent epidemic patterns in a fictitious embedding, analogous to neighborhood-level outbreak maps in real urban systems.

\begin{itemize}[label=--]
    \item In the structural map (\ref{fig:voronoi_base}), prior to disease introduction, three well-defined zones corresponding to heterogeneous communities are observed.
    \item In the epidemic map at peak time $t_1$ (\ref{fig:voronoi_peak}), infection concentrates along pathways associated with low curvature.
    \item High-curvature regions form coherent and spatially continuous clusters.
    \item Areas associated with high curvature act as partial transmission barriers, limiting spatial spread.
\end{itemize}

\begin{figure*}[h!]
\centering
\includegraphics[width=0.8\textwidth]{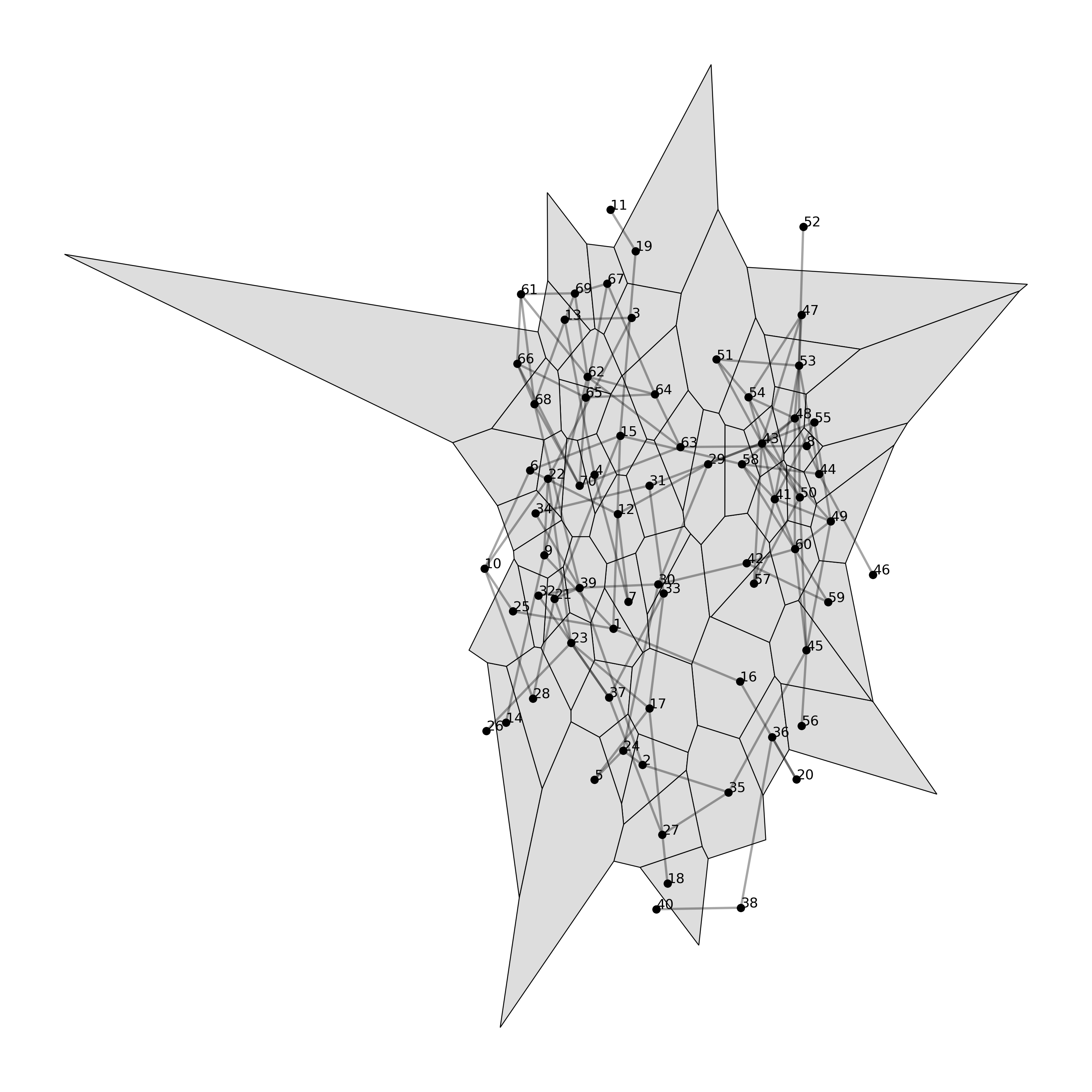}
\caption{Synthetic Voronoi map — topological structure prior to the outbreak (all regions disease-free).}
\label{fig:voronoi_base}
\end{figure*}

\begin{figure*}[h!]
\centering
\includegraphics[width=0.8\textwidth]{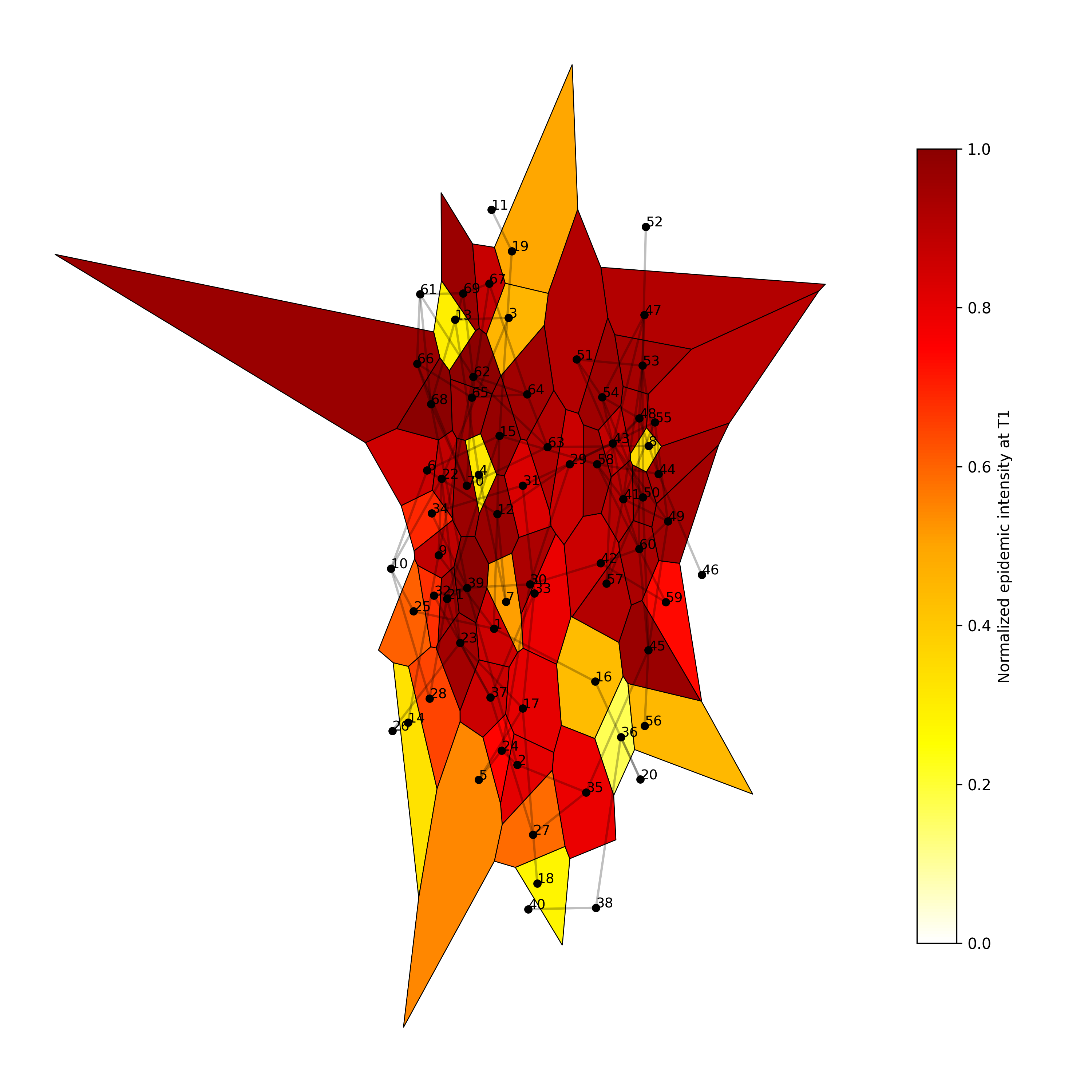}
\caption{Synthetic Voronoi map at peak time $t_1$, with infection intensity per cell.}
\label{fig:voronoi_peak}
\end{figure*}

\subsection{Structural and epidemiological interpretation}

The numerical results corroborate the central hypothesis of this work: curvature acts as a structural modulator of transmissibility on networks. More precisely:

\begin{itemize}[label=--]
    \item Curvature weighting reduces the spectral radius of $M$, thereby lowering the effective epidemic threshold
    $R_0 = \beta\, \lambda_{\max}(M)/\gamma$, in agreement with the spectral arguments developed earlier.
    \item \textbf{Low-curvature regions promote intense outbreaks}, exhibiting high communicability and prolonged persistence.
    \item \textbf{High-curvature regions dampen propagation}, behaving as natural buffers that restrict spread and accelerate epidemic dissipation.
    \item The mesostructural heterogeneity emerging in $M$ does not appear in $A$, highlighting the critical role of curvature-induced geometry.
\end{itemize}

To our knowledge, these results provide the first numerical evidence that curvature-weighted transmissibility can induce pronounced mesoscopic differentiation in otherwise well-mixed modular networks.

Overall, the findings suggest that curvature-based metrics can serve as robust structural predictors of epidemic risk, offering a principled pathway toward spatially informed surveillance strategies and refined intervention planning.

\section{Theoretical Perspectives and Directions for Future Research}

The curvature–weighted SIR model developed in this work suggests the existence
of a deep geometric structure capable of regulating epidemic dynamics on complex
networks. The combination of spectral analysis, monotone dynamical systems, and
edge–based curvature measures opens the possibility of formulating a general
theory of \emph{geometric epidemics}. Advancing such a theory will require new
mathematical results at the interface between geometry, network science, and
nonlinear epidemiological dynamics. In this section, we outline a research
program intended to serve as a foundation for future developments.

\subsection{1. Core Mathematical Structure}

The framework introduced in this article rests upon three structural pillars:

\begin{itemize}[label=--]
    \item[(i)] \textbf{Geometry–weighted contact matrix:}
    \[
        M = A \odot W(\kappa),
    \]
    where $A$ is the adjacency matrix, $\kappa$ is an edge–based curvature
    measure, and $W(\kappa)$ is a smooth positive transformation.

    \item[(ii)] \textbf{Monotone epidemiological system:}
    \[
    \begin{aligned}
    s_i' &= u(1 - s_i) - \beta\, s_i (My)_i,\\
    y_i' &= \beta\, s_i (My)_i - \gamma y_i,\\
    r_i' &= \gamma y_i - u r_i,
    \end{aligned}
    \]
    which preserves the positive cone and exhibits strong monotonicity.

    \item[(iii)] \textbf{Fixed-point operator associated with the endemic equilibrium:}
    \[
        (Ty)_i
        =
        \frac{\beta}{\gamma}
        \left(\frac{u}{\,u + \beta (My)_i\,}\right)
        (My)_i.
    \]
\end{itemize}

This structure produces a model that is simultaneously geometric, spectral, and
dynamically ordered.

\subsection{Structural Laws of Epidemic Dynamics in Curved Networks}

The theoretical analysis developed in the preceding sections allows us to
formulate three structural principles governing the interaction between network
geometry, spectral properties, and global epidemic dynamics. Two of these
principles---previously conjectural---have been rigorously established in this
work. The third remains a deep geometric hypothesis to be explored in future
research.

\begin{theorem}[Spectral reduction induced by curvature]
\label{thm:spectralreduction}
Let $A$ be the adjacency matrix of a connected graph and let
$W(\kappa)$ be a weight matrix defined by $W_{ij}=f(\kappa_{ij})$ with
$0 < f(\kappa_{ij}) \le 1$ whenever $A_{ij}=1$. Then:
\[
    \lambda_{\max}(A \odot W(\kappa))
    \;\le\;
    \lambda_{\max}(A),
\]
with equality if and only if the curvature map $f(\kappa_{ij})$ is constant on
all edges. Thus, curvature fields act as geometric regularizers of connectivity,
reducing the spectral radius and the effective epidemic threshold.
\end{theorem}

\begin{conjecture}[Geometric threshold of endemicity]
\label{conj:geothreshold_revised}
There exists a constant $\kappa^\star>0$ such that, if the mean curvature
satisfies
\[
    \overline{\kappa} > \kappa^\star,
\]
then the disease-free equilibrium is globally stable, even when the unweighted
system satisfies $R_0>1$. Sufficiently positive curvature prevents endemic
persistence by reducing the effective connectivity of the network.
\end{conjecture}

\begin{theorem}[Ordered convergence to the endemic equilibrium]
\label{thm:ordercontraction_final}
If $R_0>1$, then the fixed-point operator $T$ is strictly monotone on the
positive cone and admits a unique interior fixed point $y^\dagger$. Moreover,
for every initial condition $y_0\in (0,1]^n$,
\[
    T^k(y_0) \longrightarrow y^\dagger,
\]
with monotone convergence, meaning that $T^k(y_0)$ either increases or
decreases monotonically toward the equilibrium. This phenomenon follows from
the theory of monotone dynamical systems
\cite{Hirsch1985,Smith1995}.
\end{theorem}

These three principles synthesize the interplay between geometry, spectrum, and
asymptotic behavior in epidemic dynamics. Curvature, by regularizing
connectivity, emerges as a structural mechanism capable of altering epidemic
thresholds and shaping the spatial organization of the endemic equilibrium.

\subsection{Open Problems and Theoretical Directions}

The results established in this work open a number of structural questions that
remain unresolved:

\begin{itemize}[label=--]
    \item[(i)] \textbf{Geometric threshold of endemicity.}  
    A full proof of Conjecture~\ref{conj:geothreshold_revised} remains open.
    Possible approaches include spectral majorization, Löwner order techniques,
    or comparison principles for nonlinear operators.

    \item[(ii)] \textbf{Smooth dependence of the endemic equilibrium on curvature.}  
    For perturbations of the form
    \[
        M(\varepsilon)=A\odot W(\kappa+\varepsilon H),
    \]
    establishing existence and determining the sign of
    \[
        \frac{d y^\dagger(\varepsilon)}{d\varepsilon}
    \]
    would provide a direct quantitative link between geometric deformation of
    the network and changes in endemicity.

    \item[(iii)] \textbf{Generalizations to models with memory or stochasticity.}  
    The interaction between curvature and fractional or stochastic operators is
    largely unexplored and suggests new geometric structures for epidemic
    dynamics.
\end{itemize}

\subsection{Natural Extensions and Future Perspectives}

The proposed formulation admits natural extensions connecting curvature to
broader classes of epidemic and diffusion dynamics:

\begin{itemize}[label=--]
    \item[(i)] \textbf{Fractional models} with memory,
    \[
        D^\alpha y = \beta s(My) - \gamma y,
    \]
    where curvature may influence memory kernels in analogy with diffusion on
    manifolds of variable curvature.

    \item[(ii)] \textbf{Stochastic models} with multiplicative noise,
    \[
        dY = (\beta SY - \gamma Y)\,dt + \sigma\, dW,
    \]
    whose geometric interpretation may involve Bakry--Émery curvature and 
    log-Sobolev criteria.

    \item[(iii)] \textbf{Integro-differential or spatially continuous models},
    which approximate the limit of dense graphs or spatial meshes and connect
    graph curvature with Ricci curvature on manifolds.
\end{itemize}

\subsection{Final Considerations}

The analysis developed in this work reveals that curvature is not merely a
geometric descriptor of a network, but a structural parameter capable of
reshaping fundamental epidemic mechanisms. By embedding curvature into the
contact matrix, we obtain a diffusion kernel whose spectral, geometric and
dynamical properties determine the global behaviour of the SIR system.

The results established here provide three major contributions.  
First, we proved that curvature acts as a geometric regularizer: any positive
curvature field reduces the spectral radius of the contact matrix and therefore
decreases the effective reproduction number.  
Second, we identified explicit geometric conditions under which the disease-free
equilibrium becomes globally stable, highlighting the existence of a curvature
threshold capable of suppressing epidemic persistence.  
Third, we characterised the global convergence structure of the endemic
equilibrium when $R_0>1$, showing that the model exhibits ordered dynamics in
the sense of monotone systems.

Together, these results outline the foundations of a mathematical framework for
\emph{geometric epidemics}, in which diffusion, curvature and nonlinear
dynamics interact systematically. The theory emerging from this perspective is
not yet complete, and the open problems identified in the previous section
indicate several promising directions. Establishing a full geometric threshold
theorem, understanding how the endemic equilibrium depends smoothly on
curvature, and extending the model to fractional, stochastic, or continuous
spatial settings would significantly deepen the interplay between geometry and
epidemiology.

Beyond theoretical interest, these findings provide structural insights relevant
to empirical modelling of spatially heterogeneous diseases. Curvature captures
patterns of communication, bottlenecks and anisotropies that traditional
homogeneous weights cannot encode. As suggested by recent empirical analyses of
dengue transmission in Recife, incorporating curvature into epidemic models may
offer a principled mechanism for explaining spatial disparities, identifying
persistent hotspots, and evaluating network-based interventions.

In this sense, curvature serves as a bridge between theory and application:
a geometric parameter that not only shapes the mathematics of the model but is
also measurable from real data. This dual character opens opportunities for
future interdisciplinary work combining mathematical modelling, network
geometry, spatial epidemiology and data science.

Overall, the framework developed in this article supports the view that epidemic
dynamics on networks with intrinsic geometry can be governed by structural
principles that are spectral, geometric and dynamical in nature. The further
development of these principles may lead to a unified theory capable of
capturing both the complexity of real-world epidemics and the mathematical
regularities that underlie their global behaviour.

\end{document}